\documentclass[a4paper,11pt]{amsart}
\usepackage[latin1]{inputenc}
\usepackage[english]{babel}
\usepackage{amsthm, amssymb, amsopn, amsfonts, amstext, amsmath, stmaryrd, enumerate, color, hyperref}
\numberwithin{equation}{section}

\newcommand{\K}{\mathscr{K}}
\newcommand{\N}{\mathbb{N}}
\newcommand{\R}{\mathbb{R}}
\renewcommand{\S}{\mathscr{S}}

\newcommand{\W}{\mathscr{W}}
\newcommand{\loc}{{\rm loc}}
\newcommand{\dive}{\mbox{\normalfont div}}
\newcommand{\Hess}{{\mbox{\normalfont Hess}}}
\newcommand{\bequ}{\begin{equation}}
\newcommand{\equ}{\end{equation}}

\theoremstyle{definition}

\theoremstyle{plain}
\newtheorem{theorem}{Theorem}[section]
\newtheorem{proposition}[theorem]{Proposition}
\newtheorem{lemma}[theorem]{Lemma}

\renewcommand{\le}{\leqslant}

\renewcommand{\ge}{\geqslant}

\usepackage{mathrsfs}

\begin{document}

\title[Singular, degenerate, anisotropic equations]{Gradient bounds and 
rigidity results\\
for singular, degenerate, anisotropic \\
partial differential equations}

\author[Matteo Cozzi, Alberto Farina, Enrico Valdinoci]{
Matteo Cozzi${}^{(1,2)}$
\and
Alberto Farina${}^{(1,3)}$
\and 
Enrico Valdinoci${}^{(2,4,5)}$
}
\date{}

\maketitle

{\scriptsize \begin{center} (1) -- Laboratoire
Ami\'enois de Math\'ematique Fondamentale et Appliqu\'ee\\
UMR CNRS 7352, Universit\'e Picardie ``Jules Verne'' \\33 Rue St Leu, 
80039 Amiens (France).\\
\end{center}
\scriptsize \begin{center} (2) -- Dipartimento di Matematica
``Federigo Enriques''\\
Universit\`a
degli studi di Milano,\\
Via Saldini 50, I-20133 Milano (Italy).\\
\end{center}
\scriptsize \begin{center} (3) -- Institut ``Camille Jordan''\\
UMR CNRS 5208, Universit\'e
``Claude Bernard'' Lyon I\\
43 Boulevard du 11 novembre 1918, 69622 Villeurbanne cedex (France).\\
\end{center}
\scriptsize \begin{center} (4) --
Istituto di Matematica Applicata e Tecnologie Informatiche
``Enrico Magenes''\\
Consiglio Nazionale delle Ricerche\\
Via Ferrata 1, I-27100 Pavia (Italy).
\end{center}
\scriptsize \begin{center}(5) -- Weierstra{\ss}
Institut f\"ur Angewandte Analysis und Stochastik\\
Mohrenstra{\ss}e 39, D-10117 Berlin (Germany).
\end{center}
\bigskip
\begin{center}
E-mail addresses: matteo.cozzi@unimi.it,
alberto.farina@u-picardie.fr,
enrico@math.utexas.edu
\end{center}
}
\bigskip
\bigskip

{\sc Abstract:}
We consider the Wulff-type energy functional
$$\W_\Omega(u) := \int_\Omega B( H( \nabla u (x) ) ) - F(u(x)) \, dx,$$
where~$B$ is positive, monotone and convex, and $H$
is positive homogeneous of degree~$1$.
The critical points of this functional satisfy a possibly singular or
degenerate, quasilinear equation in an anisotropic medium.

We prove that the gradient of the solution is bounded
at any point by the potential $F(u)$ and we deduce several rigidity and
symmetry properties.
\bigskip
\bigskip

\section{Introduction and main results}

We consider here a variational problem in an anisotropic medium. The 
physical motivation we have in mind comes from some well-established 
models of surface energy, see for instance~\cite{T78, G06}
and references 
therein for a classical introduction to the topic.

Surface energy arises since the microscopic environment
of the interface of a medium is different from
the one in the bulk of the substance.
In many concrete cases, such as for the common
cooking salt, the different behavior depends
significantly on the space direction and
so these anisotropic surface energies have now become very
popular in metallurgy and crystallography, see e.g.~\cite{W01, D44, AC77}.
Applications to
crystal growth and thermodynamics
are discussed in~\cite{MBK77, C84, TCH92} and in~\cite{G93}, respectively.

Other applications of related anisotropic models occur in
noise-removal procedures in digital image processing,
crystalline mean curvature flows and crystalline fracture theory,
see e.g. \cite{NP99, BNP01a, BNP01b, EO04, OBGXY05}
and references therein.
See also~\cite{FM91, C04}
for anisotropic problems related to the 	
Willmore functional and~\cite{CS09, WX11} for
elliptic anisotropic
systems inspired by fluidodynamics. We defer the interested reader to Appendix~\ref{wulshaapp} for some deeper physical insights.

Of course, besides this surface energy, the medium may also be subject to exterior forces and the total energy functional is in this case the sum of an anisotropic surface energy plus a potential term. More precisely, the mathematical framework we work in
is inspired by the Wulff crystal construction (see pages~571--573 in~\cite{T78}) and it may be formally introduced as follows.

Given a domain~$\Omega \subseteq \R^n$, with~$n \ge 2$, consider the functional
\bequ \label{wulfun}
\W_\Omega(u) := \int_\Omega B( H( \nabla u (x) ) ) - F(u(x)) \, dx.
\equ
Here,~$B$ denotes a function of class~$C_\loc^{3, \beta}((0, +\infty))\cap C^1([0, +\infty))$, with~$\beta \in (0, 1)$, such that~$B(0) =  B'(0) = 0$ and
\bequ \label{Bsign}
B(t), B'(t), B''(t) > 0 \mbox{ for any } t \in (0, +\infty).
\equ
Also,~$H: \R^n \to \R$ is a positive homogeneous function of degree~$1$, of class~$C_\loc^{3, \beta}(\R^n \setminus \{ 0 \})$, with
\bequ \label{Hsign}
H(\xi) > 0 \mbox{ for any } \xi \in \R^n \setminus \{ 0 \}.
\equ
Notice that, being~$H$ homogeneous and defined at the origin, it necessarily holds~$H(0) = 0$. 
Finally, take~$F \in C^{2,\beta}_\loc(\R)$ and assume that either~(A) or~(B) is 
satisfied, where:
\medskip

\begin{enumerate}[(A)]
\item There exist~$p > 1$,~$\kappa \in [0, 1)$ and positive~$\gamma, \Gamma$ such that, for any~$\xi \in \R^n \setminus \{ 0 \}$, $\zeta \in \R^n$,
$$
\left[ \Hess \,(B \circ H)(\xi) \right]_{i j} \zeta_i \zeta_j \ge \gamma {(\kappa + |\xi|)}^{p - 2} {|\zeta|}^2, \label{BHpell}
$$
and
$$
\sum_{i, j = 1}^n \left| \left[ \Hess \,(B \circ H)(\xi) \right]_{i j} \right| \le \Gamma {(\kappa + |\xi|)}^{p - 2}.
$$
\item The composition~$B \circ H$ is of class~$C^{3,\beta}_\loc(\R^n)$ 
and for any~$K>0$
there exist a positive constant~$\gamma$ such that,
for any~$\xi, \zeta \in \R^n$, with~$|\xi|\le K$, we have
\begin{equation*}
\left[ \Hess \,(B \circ H)(\xi) \right]_{i j} \zeta_i \zeta_j \ge \gamma 
\,{|\zeta|}^2.
\end{equation*}
\end{enumerate}
\medskip

Here above and  
throughout the paper, the
summation convention for
repeated subscripts is used, unless differently specified.
Critical points of~$\W_\Omega$ weakly satisfy the Euler-Lagrange equation
\bequ \label{eleq}
\frac{\partial}{\partial x_i} 
\Big( B'( H(\nabla u) ) H_i(\nabla u) \Big) + F'(u) = 0,
\equ
where~$H_i(\xi) = \partial_{\xi_i} H(\xi)$. 

The model we consider is indeed very general and
it allows at the same time
an anisotropic dependence on the space variable
and a possible singularity
or degeneracy of the diffusion operator. For instance,
we can take into account the following
examples of~$B$:
\bequ \label{EXAM}
B(t) = \frac{{\left(\kappa^2 + t^2\right)}^{p/2} - \kappa^p}{p} \qquad
\mbox{ and } \qquad B(t) = \sqrt{1 + t^2} - 1,
\equ
with~$p > 1$, and $\kappa \ge 0$.

Such choices are related to
the {\it anisotropic~$p$-Laplace equation}
\bequ\label{001}
\dive \left( H^{p - 1}(\nabla u) \nabla H(\nabla u) \right) + F'(u) = 0,
\equ
obtained by taking~$B(t) = t^p/p$, and the {\it anisotropic minimal surface equation}
\bequ\label{002}
\dive \left( \frac{H(\nabla u) \nabla H(\nabla u)}{\sqrt{1 + H^2(\nabla u)}} \right) + F'(u) = 0.
\equ
In particular, when~$H(\xi)=|\xi|$, equations~\eqref{001}
and~\eqref{002} reduce respectively
to the classical $p$-Laplace
and minimal surface equations.

As an example of anisotropic~$H$ one may consider the function
\bequ \label{HA}
H(\xi) = \sqrt{\langle M \xi, \xi \rangle},
\equ
with~$M \in \mbox{\normalfont Mat}_n(\R)$ symmetric and positive definite.
We stress that the combination of such a~$H$ along with~$B$ as
in~\eqref{EXAM}, with~$\kappa > 0$ in the~$p$-Laplacian case, actually produces
an operator that satisfies hypothesis~(B). In Appendix~\ref{Hchar} we prove
that indeed this is the only possible choice for~$H$, in the framework of
assumption~(B).

We refer instead to Appendix~\ref{notanorm} for the construction of a rather general
anisotropic function~$H$ which is not necessarily a norm.

Given~$u: \R^n \to \R$, we define
$$
c_u := \sup \left\{ F(r) : r \in \left[ \inf_{\R^n} u, \sup_{\R^n} u \right] \right\}.
$$
The quantity~$c_u$ is an important potential gauge.
Indeed, the nonlinearity~$f$ defines the potential~$F$
uniquely up to an additive constant. An appropriate
choice of this constant makes the results that we are going
to present as sharp as possible: roughly speaking,
this gauge consists
in taking~$F(u)-c_u$ as effective potential
(notice that such potential is non-positive on the solution~$u$).
Furthermore we are able to identify explicitly the value
of the quantity~$c_u$, as showed in Theorem~\ref{liou2}.
 
Our main results are a pointwise estimate on the gradient of
the solution, from which we deduce some rigidity and symmetry properties
(in particular, we obtain one-dimensional Euclidean symmetry 
and Liouville-type results).

Thus, the first result we present is a pointwise bound
on the gradient in terms of the effective potential.
Notice that classical elliptic estimates provide bounds of the gradient
in either H\"older or Lebesgue norms, but do not give any pointwise
information in general.
In dimension~$1$, the pointwise estimate that we
present reduces to the classical Energy Conservation Law.

In higher dimension, estimates of this kind were given first by~\cite{M85}
for the semilinear equation
$$\Delta u+F'(u)=0$$ with~$F\le0$
(this case is comprised in our setting
by choosing~$H(\xi)=|\xi|$, $B(t)=t^2/2$). Then,~\cite{CGS94}
extended such estimates to the quasilinear case
$$\dive (\Phi'(|\nabla u|^2)\nabla u)+F'(u)=0$$
with~$F\le0$
(this is a particular case in our framework given by~$H(\xi)=|\xi|$, $B(t)=(1/2)\Phi(t^2)$).

Recently, some attention has been given to the case of anisotropic
media and the first pointwise estimate in this setting was given in~\cite{FV13}
for equations of the type
$$ \dive (H(\nabla u)\nabla H(\nabla u))+F'(u)=0$$
(again, this is a particular case for us by taking~$B(t)=t^2/2$). 

Our purpose
is to extend the previous results to the general case of anisotropic media
with possible nonlinearities, singularities and nondegeneracies
in the diffusion operator (indeed, the function~$H$ encodes
the anisotropy of the medium and the function~$B$ the possible
degeneracies of the operator). The precise statement of our
pointwise bound is the following:

\begin{theorem} \label{gradpest}
Assume that one of the following conditions is valid:
\begin{enumerate}[(i)]
\item Assumption~(A) holds and~$u \in L^\infty(\R^n) \cap W_\loc^{1, p}(\R^n)$ is a weak solution of~\eqref{eleq} in~$\R^n$;
\item Assumption~(B) holds and $u \in W^{1, \infty}(\R^n)$ weakly solves~\eqref{eleq} in~$\R^n$.
\end{enumerate}
Then, for any~$x \in \R^n$,
\bequ \label{gradpesteq}
B'(H(\nabla u(x))) H(\nabla u(x)) - B(H(\nabla u(x))) \le c_u - F(u(x)).
\equ
Moreover, if there exists~$x_0 \in \R^n$ such that
$$
\nabla u(x_0) \ne 0
$$
and
\bequ \label{gradpideq}
B'(H(\nabla u(x_0))) H(\nabla u(x_0)) - B(H(\nabla u(x_0))) = c_u - F(u(x_0)),
\equ
then
\bequ \label{gradpideq2}
B'(H(\nabla u)) H(\nabla u) - B(H(\nabla u)) = c_u - F(u).
\equ
on the whole connected component of~$\{ \nabla u \ne 0 \}$ containing~$x_0$.
\end{theorem}

Now we state our main symmetry result, according to which
the equality in~\eqref{gradpideq} implies that the solution
only depends on one Euclidean variable (in particular,
the classical and anisotropic curvatures of the level sets
vanish identically):

\begin{theorem}\label{1D}
Let~$u$ be as in
Theorem~\ref{gradpest}. Suppose that
there exists~$x_0 \in \R^n$ such that~$\nabla u(x_0) \ne 0$
and~\eqref{gradpideq} holds true.

Then
there exist~$u_0:\R\rightarrow\R$ and~$\omega\in S^{n-1}$ such 
that~$u(x)=
u_0(\omega\cdot x)$ for any~$x$ in the connected component~${\mathscr{S}}$
of~$\{\nabla u\ne0\}$ containing~$x_0$, and
the level sets of $u$ in~${\mathscr{S}}$
are affine hyperplanes.
\end{theorem}

We observe that one-dimensional solutions $u(x)=u_0(\omega\cdot x)$
of~\eqref{eleq}
satisfy the ordinary differential equation
\bequ\label{ODE}
B''(H(\omega \dot u_0))\,H^2(\omega)\,\ddot u_0+F'(u_0)=0.\equ
Equivalently, \eqref{gradpideq2} reduces in this case to
the Energy Conservation Law
$$ b(H(\omega \dot u_0))=c_{u_0}-F(u_0),$$
where~$b(t):=B'(t)t-B(t)$.

Theorem~\ref{1D} was proved in the isotropic setting in~\cite{CGS94}
under the additional assumption that~$F\le0$,
and in the planar, anisotropic setting in~\cite{FV13}. Therefore
Theorem~\ref{1D} is new in the anisotropic setting
even for cases that are not singular or degenerate
(e.g. for $B(t)=t^2/2$).
We stress in particular that
the proof of this result is different from
the ones in~\cite{CGS94, FV13} since we exploit for
the first time the consequences of the vanishing of the~$P$-function
by taking into account explicitly an appropriate remainder term:
indeed, such $P$-function is not only a subsolution
of a suitable equation, but it is also a solution of
an equation with a term that has a sign and that vanishes
when $P$ is constant (see the forthcoming equation~\eqref{Pine}
for details).

Under some further (but natural) assumptions,
Theorem~\ref{1D} holds globally in the whole of the space,
as next results point out:

\begin{theorem}\label{1DAUX2}
Let~$u$ be as in
Theorem~\ref{gradpest} with condition~(ii) in force.
Assume 
that there exists~$x_0 \in \R^n$ such that~$\nabla u(x_0) \ne 0$
and~\eqref{gradpideq} holds true.

Then
there exist~$u_0:\R\rightarrow\R$ and~$\omega\in S^{n-1}$ such
that~$u(x)=
u_0(\omega\cdot x)$ for any~$x\in\R^n$.
\end{theorem}

We observe that the assumptions of Theorem~\ref{1DAUX2}
are satisfied by many cases of interest, such as
the minimal surface and the regularized~$p$-Laplace equations
(e.g. with~$B$ as in~\eqref{EXAM} with~$\kappa>0$).
A global version of Theorem~\ref{1DAUX2} which encompasses
all the cases under consideration
is given by the following result:

\begin{theorem}\label{1DAUX}
Let~$u$ be as in
Theorem~\ref{gradpest} and assume that~\eqref{gradpideq2}
holds in the whole of~$\R^n$.

Then
there exist~$u_0:\R\rightarrow\R$ and~$\omega\in S^{n-1}$ such
that~$u(x)=
u_0(\omega\cdot x)$ for any~$x\in\R^n$.
\end{theorem}

Differently from~\cite{CGS94} in which results similar
to Theorems~\ref{1DAUX2} and~\ref{1DAUX} were
obtained in the isotropic setting
with a different method, we do not need to
assume any sign assumption on~$F$.
The next is a Liouville-type result that shows that the solution is constant
if the effective potential and its derivative vanish at some point
(the isotropic case was dealt with in~\cite{CGS94, CFV12}).

\begin{theorem} \label{liou}
Let~$u$ be as in Theorem~\ref{gradpest}. If condition~(i) of Theorem~\ref{gradpest} is in force, with~$\kappa = 0$ and~$p > 2$, assume in addition that, given a value~$r \in \R$ such that~$F(r) = c_u$ and~$F'(r) = 0$, we have
\bequ \label{Fgrowliou}
|F'(\sigma)| = O({|\sigma - r|}^{p - 1}) \mbox{ as } \sigma \rightarrow r.
\equ
If there exists a point~$x_0 \in \R^n$ for which $F(u(x_0)) = c_u$ and~$F'(u(x_0)) = 0$, then~$u$ is constant.
\end{theorem}

Notice that condition~\eqref{Fgrowliou}
cannot be removed from Theorem~\ref{liou}, since, without such assumption,
one can construct smooth, non-constant, one-dimensional solutions:
see Proposition~7.2
in~\cite{FSV08} for an explicit, non-constant example in which~\eqref{Fgrowliou}
is not satisfied and
$$
F \left( \min_{\R^n} u \right)=F \left( \max_{\R^n} u \right)=c_u
\ {\mbox{ and }} \
F '\left( \min_{\R^n} u \right)=F '\left( \max_{\R^n} u \right)=0.$$
We also remark that, in principle, to obtain~$c_u$ one is expected
to know all the values of the solution~$u$ and to compute the potential out of them.
Next result shows in fact that this is not necessary, and that~$c_u$
may be computed once we know only the infimum and the supremum of the solution
(the isotropic case was dealt with in~\cite{FV10}):

\begin{theorem} \label{liou2}
Let~$u$ and~$F$ be as in Theorem~\ref{liou}. Then,
$$
c_u = \max \left\{ F \left( \inf_{\R^n} u \right), F \left( \sup_{\R^n} u \right) \right\}.
$$
Furthermore, if there exists~$y_0 \in \R^n$ such that~$F(u(y_0)) = c_u$, then 
$$ {\mbox{either~$u(y_0) = \displaystyle\inf_{\R^n} u$ or~$u(y_0) = 
\displaystyle\sup_{\R^n} u$.}}$$
\end{theorem}

The paper is organized as follows. First, in Section~\ref{S:1}
we collect some technical and ancillary results. The regularity of the solutions
is briefly tackled in Section~\ref{S:2}. The proof of Theorem~\ref{gradpest}
relies on a $P$-function argument that is discussed in Section~\ref{S:3}
(roughly speaking, one has to check that a suitable energy function
is a subsolution of a partial differential equation and to use
the Maximum Principle to obtain the desired bound).
The proofs of the main results are collected in
Sections~\ref{S:4}--\ref{S:8}. In Appendices~\ref{notanorm} and~\ref{Hchar}, respectively, we present
an example of function~$H$ which is not a norm and the proof of the fact that any~$H$ fulfilling assumption~(B) is of the form~\eqref{HA}. Finally, some physical interpretations of the so-called Wulff shape of~$H$ are briefly discussed in Appendix~\ref{wulshaapp}.

\section{Some preliminary results}\label{S:1}

The first part of
this section is mainly devoted to some elementary facts about positive homogeneous functions.
We mostly provide only the statements, referring to \cite{FV13} for the omitted proofs.

We recall that a function~$H: \R^n \setminus \{ 0 \} \to \R$ is said to be positive homogeneous of degree~$d \in \R$ if $H(t \xi) = t^d H(\xi)$, for any~$t > 0$ and~$\xi \in \R^n \setminus \{ 0 \}$.

\begin{lemma} \label{iter}
If~$H\in C^m(\R^n\setminus \{0\})$ is positive homogeneous of
degree~$d$ and~$\alpha\in \N^n$ with~$\alpha_1+\dots+\alpha_n=m$,
then~$\partial^\alpha H$ is positive homogeneous of degree~$d-m$.
\end{lemma}

Notice that the corresponding result proved in \cite{FV13}, Lemma 2, only deals with integer 
degrees. Nevertheless, the proof works the same way considering a real degree $d$.

Next, we establish the identities commonly used in the course of the 
main proofs.

\begin{lemma} \label{homain}
If~$H\in C^3(\R^n\setminus \{0\})$ is positive homogeneous of
degree~$1$, we have that
\begin{align}
\label{i} H_i(\xi) \xi_i &= H(\xi),\\
\label{ii} H_{ij}(\xi) \xi_i &= 0,\\
\label{iii} H_{ijk}(\xi) \xi_i &=  -H_{jk}(\xi).
\end{align}
\end{lemma}

Now, we justify the smoothness of~$H$ needed to write \eqref{eleq}
and to use the regularity theory:

\begin{lemma} \label{derBH0}
Let~$H\in C^1(\R^n\setminus\{0\})$ be a positive homogeneous function of degree $d$ admitting non-negative values and~$B \in C^1([0, +\infty))$, with~$B(0) = 0$. Assume that either $d > 1$ or $d = 1$ and $B'(0) = 0$. Then~$H$ can be extended by setting~$H(0):=0$ to a continuous function, such that~$B \circ H \in C^1(\R^n)$ and
$$
\partial_i (B \circ H)(0) = 0 = \lim_{x \rightarrow 0} B'(H(x)) H_i(x).
$$
\end{lemma}

\begin{proof} 
Setting~$H(0) := 0$ clearly transforms~$H$ into a continuous function on 
the whole 
of~$\R^n$, since~$|H(\xi)|\le |\xi|^d \sup_{{\rm S}^{n-1}} |H|$, for any~$\xi \in \R^n \setminus \{ 0 \}$. Moreover,~$B \circ H \in
C^1(\R^n\setminus\{0\})$, and
\begin{equation*}
\begin{split}
\partial_i (B \circ H)(0) & = \lim_{t\rightarrow 0} \frac{B(H(te_i))}{t} = \lim_{t \rightarrow 0^\pm} \frac{B(H(\pm |t| e_i))}{t} \\
& = \lim_{t \rightarrow 0^\pm} \frac{B(|t|^d H(\pm e_i))}{t} = \pm {H(\pm e_i)}^{\frac{1}{d}} \lim_{s \rightarrow 0^+} \frac{B(s)}{s} s^{\frac{d - 1}{d}} \\
& = \pm {H(\pm e_i)}^{\frac{1}{d}} B'(0) \lim_{s \rightarrow 0^+} s^{\frac{d - 1}{d}} = 0.
\end{split}
\end{equation*}
On the other hand, by Lemma~\ref{iter}, $H_i(x)=|x|^{d-1} H_i\left( x / |x| \right)$ for any~$x\in\R^n\setminus\{0\}$, and so
\begin{equation*}
\begin{split}
\lim_{x \rightarrow 0} |B'(H(x)) H_i(x)| & \le \sup_{S^{n - 1}} |H_i| \lim_{x \rightarrow 0} |x|^{d-1} \left| B' \left( |x|^d H \left( \frac{x}{|x|} \right) \right) \right| \\
& = |B'(0)| \sup_{S^{n - 1}} |H_i| \lim_{x \rightarrow 0} |x|^{d-1} = 0,
\end{split}
\end{equation*}
as desired.
\end{proof}

Then,
we have the following characterization of the positive definiteness of the composition~$B \circ H$.

\begin{lemma} \label{WXgen}
Let~$B \in C^2((0, +\infty))$ be a function satisfying~\eqref{Bsign} and~$H \in C^2(\R^n \setminus \{ 0 \})$ be positive homogeneous of degree~$1$ satisfying~\eqref{Hsign}. Then, the following two statements are equivalent:
\begin{enumerate}[(i)]
\item $\Hess \,(B \circ H)$ is positive definite in~$\R^n \setminus \{ 0 \}$;
\item The restriction of~$\Hess \, (H) (\xi)$ to~$\xi^\perp$ is a positive definite endomorphism~$\xi^\perp \to \xi^\perp$, for all~$\xi \in S^{n - 1}$.
\end{enumerate}
\end{lemma}

\begin{proof}
Our argument is an adaptation of the proof of Proposition 2 on page~102 of~\cite{WX11}. The case covered there is the one with~$B(t) = t^2$.\\
First, we prove that~$(i)$ implies~$(ii)$. Fix~$\xi \in S^{n - 1}$. Assumption~$(i)$ is equivalent to
\bequ \label{pposdefexp}
\left[ B''(H(\xi)) H_i(\xi) H_j(\xi) + B'(H(\xi)) H_{i j}(\xi) \right] \zeta_i \zeta_j > 0 \mbox{ for any } \zeta \in \R^n \setminus \{ 0 \}.
\equ
Observe now that~$\nabla H(\xi)$ cannot be orthogonal to~$\xi$, since, by~\eqref{i},~$H_i(\xi) \xi_i = H(\xi) > 0$. Therefore,~${\nabla H(\xi)}^\perp$ and~$\xi$ span the whole of~$\R^n$. Letting now~$V \in \xi^\perp$, we write
$$
V = \zeta + \lambda \xi, \qquad \mbox{for some } \zeta \in {\nabla H(\xi)}^\perp \setminus \{ 0 \}, \lambda \in \R.
$$
Applying~\eqref{pposdefexp} with~$\zeta = V - \lambda \xi$ and using~\eqref{ii}, we get
\begin{equation*}
\begin{split}
0 & < \left[  B''(H(\xi)) H_i(\xi) H_j(\xi) + B'(H(\xi)) H_{i j}(\xi) \right] \zeta_i \zeta_j = B'(H(\xi)) H_{i j}(\xi) \zeta_i \zeta_j \\
& = B'(H(\xi)) H_{i j}(\xi) (V_i - \lambda \xi_i) (V_j - \lambda \xi_j)= B'(H(\xi)) H_{i j}(\xi) V_i V_j,
\end{split}
\end{equation*}
which, by \eqref{Bsign}, gives~$(ii)$.\\
Conversely, assume that~$(ii)$ holds. Let~$V \in \R^n \setminus \{ 0 \}$ and decompose it into~$V = \eta + \lambda \xi$, for~$\eta \in \xi^\perp$,~$\lambda \in \R$. By~\eqref{ii},~\eqref{Bsign} and~$(ii)$ we obtain
\begin{eqnarray*}
&& \left[ B''(H(\xi)) H_i(\xi) H_j(\xi) + B'(H(\xi)) H_{i j}(\xi) \right] V_i V_j \\
&& \qquad = B''(H(\xi)) H_i(\xi) H_j(\xi) V_i V_j + B'(H(\xi)) H_{i j}(\xi)(\eta_i + \lambda \xi_i) (\eta_j + \lambda \xi_j) \\
&& \qquad = B''(H(\xi)) {\left[ V \cdot \nabla H(\xi) \right]}^2  + B'(H(\xi)) H_{i j}(\xi) \eta_i \eta_j \ge B'(H(\xi)) H_{i j}(\xi) \eta_i \eta_j > 0,
\end{eqnarray*}
if~$\eta \ne 0$. If on the other hand~$\eta = 0$, i.e.~$V = \lambda \xi$ with~$\lambda \ne 0$, then, using~\eqref{i} and~\eqref{ii},
\begin{eqnarray*}
&& \left[ B''(H(\xi)) H_i(\xi) H_j(\xi) + B'(H(\xi)) H_{i j}(\xi)  \right] V_i V_j \\
&& \qquad = \lambda^2 \left[ B''(H(\xi)) H_i(\xi) H_j(\xi) \xi_i \xi_j + B'(H(\xi)) H_{i j}(\xi) \xi_i \xi_j \right] = \lambda^2 B''(H(\xi)) H^2(\xi) > 0,
\end{eqnarray*}
so that~$(i)$ is proved.
\end{proof}

Next, we have a result ensuring the convexity of~$H$. We point out that this actually comes as a corollary of Lemma~\ref{WXgen} and Lemma~\ref{homain} together.

\begin{lemma} \label{L990}
Let~$H\in C^2(\R^n\setminus \{0\})$ be a positive homogeneous function of degree~$1$ satisfying~\eqref{Hsign} and~$B \in C^2((0, +\infty))$ be such~\eqref{Bsign} holds. Assume also~$\Hess \,(B \circ H)$ to be positive definite in~$\R^n \setminus \{ 0 \}$. Then~$H$ is convex and
\begin{equation}\label{convex}
H_{ij}(\xi)\,\eta_i\eta_j\ge0\qquad{\mbox{
for any~$\xi\in\R^n\setminus\{0\}$ and~$\eta\in\R^n$.}}
\end{equation}
\end{lemma}

Following is a linear algebra result
that is crucial for the subsequent proofs of 
Proposition~\ref{SP:SP} and
Theorem~\ref{1D}.

\begin{proposition}\label{corpos}
Let $H$ and $B$ as in the statement of Lemma \ref{L990}. Then, given any 
matrix~$\{c_{i j}\}_{i,j\in\{1,\dots,n\}}$, we have
\begin{equation}\label{EQ0}
H_{i j}(\xi) H_{k \ell}(\xi) c_{i k} c_{j \ell} \ge 0 \qquad \mbox{for any 
} \xi \in \R^n \setminus \{ 0 \}.\end{equation}
Moreover, assume that equality holds in~\eqref{EQ0}
for a vector~$\xi=(\xi_1,\dots,\xi_n)\in\R^n\setminus\{0\}$ such that
\begin{equation}\label{EQ1}
\xi_1= \dots = \xi_{n-1}=0.\end{equation}
Then\footnote{To avoid
confusion, we use indices like~$i$
ranging in~$\{1,\dots,n\}$ and like~$i'$
ranging in~$\{1,\dots,n-1\}$.}
\begin{equation}\label{2.14a}
{\mbox{$c_{i' j'}=0$
for any~$i', j'\in\{1,\dots,n-1\}$.}}\end{equation}
\end{proposition}

\begin{proof}
We follow the argument given at the end of the proof of Proposition 1 of \cite{FV13}. 
By points~$(ii)$ in Lemma~\ref{WXgen}
and~\eqref{ii}, we know that
\begin{equation}\label{EQ3}\begin{split}&
{\mbox{$\Hess(H)(\xi)$
has~$n-1$ strictly positive eigenvalues and
one null eigenvalue}}\\
&{\mbox{(the latter corresponding to the eigenvector~$\xi$).}}
\end{split}\end{equation}
Therefore, we can
diagonalize it via an orthogonal matrix~$\{M_{ij}\}_{i,j\in
\{1,\dots,n\}}$, by 
writing
\begin{equation}\label{EQ2}
{\mbox{$H_{i j} = M_{p i} \lambda_p M_{p j}$, with $\lambda_1\ge\dots\ge
\lambda_{n-1}>
\lambda_n =0$.}}\end{equation} So, setting 
\begin{equation}\label{EQK}
\vartheta_{p r} := M_{p i} M_{r m} c_{i m},
\end{equation}
for fixed $p$ and $r$, we have that
$$
0 \le {(\vartheta_{p r})}^2 = (M_{p i} M_{r k} c_{i k}) (M_{p j} M_{r \ell} c_{j \ell}) = M_{p i} M_{p j} M_{r k} M_{r \ell} c_{i k} c_{j \ell}.
$$
Now, multiply by $\lambda_p \lambda_r$ and sum over $p$ and $r$. We get
\begin{equation}\label{EQQ}
0 \le \lambda_p\lambda_r {(\vartheta_{p r})}^2=
M_{p i} \lambda_p M_{p j} M_{r k} \lambda_r M_{r \ell} c_{i k} c_{j \ell} = H_{i j} H_{k \ell} c_{i k} c_{j \ell},
\end{equation}
which proves~\eqref{EQ0}.

Now we assume~\eqref{EQ1} and we
suppose that equality holds 
in~\eqref{EQ0}. We claim that
\begin{equation}\label{EQ2.2}
{\mbox{$M_{n i'}=0$
for any $i'\in \{1,\dots,n-1\}$.}}
\end{equation}
For this, we use a classical linear algebra
procedure: we define~$w_i:=M_{ni}$ and we 
consider the 
vector~$w=(w_1,\dots,w_n)$. We exploit~\eqref{EQ2}
and we have, for 
any~$j\in\{1,\dots,n\}$,
\begin{eqnarray*}
\big(\Hess(H)(\xi)w\big)_j = 
H_{jk}w_k=
M_{ij} \lambda_i M_{ik} w_{k}=
M_{ij} \lambda_i M_{ik} M_{nk}\\
=M_{ij} \lambda_i \delta_{in}=M_{nj}\lambda_n=0=(0\, w)_j. 
\end{eqnarray*}
That is,~$w$ is an eigenvector
for~$\Hess(H)(\xi)$ and so, 
by~\eqref{EQ3}, $w$ is parallel to~$\xi$. Thus, by~\eqref{EQ1},
$w$ is parallel to~$(0,\dots,0,1)$ and so~$w_{i'}=0$
for any~$i'\in \{1,\dots,n-1\}$, proving~\eqref{EQ2.2}.

Now, if equality holds in~\eqref{EQ0}, then~\eqref{EQQ} gives that
$$ 0=\lambda_p\lambda_r {(\vartheta_{p r})}^2.$$
Consequently, by~\eqref{EQ2}, we obtain that
\begin{equation}\label{EQ55}
{\mbox{$\vartheta_{p' r'}=0$
for any~$p',r'\in\{1,\dots,n-1\}$.}}\end{equation}
Hence, we invert~\eqref{EQK}
and we obtain that
$$ M_{pj} M_{rk} \vartheta_{p r} = 
M_{p j} M_{p i} M_{r k} M_{r m} c_{i m}
=\delta_{ij}\delta_{mk} c_{im}=c_{jk}$$
for any~$j,k\in\{1,\dots,n\}$.
So, recalling~\eqref{EQ2.2} and~\eqref{EQ55}, we have,
for any~$j',k'\in\{1,\dots,n-1\}$,
$$ c_{j'k'}=
M_{pj'} M_{rk'} \vartheta_{p r}
=M_{p'j'} M_{r'k'} \vartheta_{p' r'}=0,$$
where the indices~$p',r'$ are summed over~$\{1,\dots,n-1\}$.
\end{proof}

Now we collect two technical inequalities concerning function~$B$ which will be used in the proofs of Theorems~\ref{1DAUX} and~\ref{liou}.

\begin{lemma} \label{blem}
Let~$B \in C^2((0, +\infty)) \cap C^0([0, +\infty))$ be a function satisfying~$B(0) = 0$ and~\eqref{Bsign}. Then,
\bequ \label{bpos}
B'(t) t - B(t) > 0,
\equ
for any~$t > 0$.
\end{lemma}

\begin{proof}
For any~$t > 0$ set
\bequ \label{bdef}
b(t) := B'(t) t - B(t).
\equ
Clearly,~$b \in C^1((0, +\infty))$. By differentiation we get, for~$t > 0$,
$$
b'(t) = B''(t) t + B'(t) - B'(t) = B''(t) t > 0,
$$
since~$B''(t)$ is positive. Thus,~$b$ is strictly increasing and so
$$
b(t) > b(0^+) = 0, \qquad \mbox{ for any } t > 0,
$$
which proves the lemma.
\end{proof}


\begin{lemma} \label{blem2}
Let~$B \in C^2((0, +\infty)) \cap C^1([0, +\infty))$ be such~$B(0) = 0$ and~$H \in C^2(\R^n \setminus \{ 0 \})$ be a positive homogeneous function of degree~$1$ satisfying~\eqref{Hsign}. Assume that they either satisfy~(A) or~(B). Then, for any~$M > 0$, there exists~$\varepsilon > 0$ such that
\bequ \label{bpos2}
B'(t) t - B(t) \ge \varepsilon t^{p^*} \mbox{ for any } t \in [0, M],
\equ
where
$$
p^* =
\begin{cases}
p & \mbox{ if (A) holds with } \kappa = 0 \\
2 & otherwise.
\end{cases}
$$
\end{lemma}

\begin{proof}
Let~$M > 0$,~$b$ be as in~\eqref{bdef} and~$\varepsilon > 0$ to be determined later. Define, for any non-negative~$t$,
$$
E(t) := b(t) - \varepsilon t^{p^*} = B'(t) t - B(t) - \varepsilon t^{p^*}.
$$
If we prove that
\bequ \label{b.1}
E'(t) \ge 0 \mbox{ for any } t \in (0, M],
\equ
is true, then we are done, since in this case we have
$$
E(t) \ge E(0) = 0 \mbox{ for any } t \in (0, M],
$$
which leads directly to~\eqref{bpos2}. To show that~\eqref{b.1} holds, fix~$t \in (0, M]$ and choose~$\xi \in \R^n \setminus \{ 0 \}$ in a way that~$t = H(\xi)$. Notice that this can surely be done by taking~$\xi := t H^{-1}(\nu) \nu$, for any~$\nu \in \R^n \setminus \{ 0 \}$. In particular, we have that
\bequ \label{b.2}
|\xi| = H^{-1} \left(\frac{\xi}{|\xi|}\right) t \le h t \le h M,
\equ
if we set~$h^{-1} = \inf_{S^{n - 1}} H > 0$. Applying now the first formula of~(A) with~$\zeta = \xi$ we get
\begin{equation*}
\begin{split}
c_1 {(\kappa + |\xi|)}^{p - 2} {|\xi|}^2 & \le  \left[ \Hess \,(B \circ H)(\xi) \right]_{i j} \xi_i \xi_j = \partial_{\xi_j} \left( B'(H(\xi)) H_i(\xi) \right) \xi_i \xi_j \\
& = \left( B''(H(\xi)) H_i(\xi) H_j(\xi) + B'(H(\xi)) H_{i j}(\xi) \right) \xi_i \xi_j \\
& = B''(H(\xi)) H^2(\xi),
\end{split}
\end{equation*}
where in the last equality we used~\eqref{i} and~\eqref{ii}. By the homogeneity of~$H$ we thus may conclude that
\bequ \label{b.3}
B''(H(\xi)) \ge c_1 {(\kappa + |\xi|)}^{p - 2} {|\xi|}^2 H^{-2}(\xi) \ge c {(\kappa + |\xi|)}^{p - 2},
\equ
for some positive constant~$c$. Now, if~$\kappa > 0$ and~$p \ge 2$, we drop~$|\xi|$ in the last parenthesis, getting
$$
B''(H(\xi)) \ge c {\kappa}^{p - 2}.
$$
If~$\kappa > 0$ but~$1 < p < 2$, then by~\eqref{b.2} we have
$$
B''(H(\xi)) \ge c {(\kappa + h M)}^{p - 2}. 
$$
If on the other hand~$\kappa = 0$, we simply rewrite~\eqref{b.3}, obtaining
$$
B''(H(\xi)) \ge c {|\xi|}^{p - 2} \ge c' H^{p - 2}(\xi),
$$
for some positive constant~$c'$. Collecting these three cases and
making explicit the dependence on~$t$, we get
\bequ \label{5.4}
B''(t) \ge c'' t^{p^* - 2},
\equ
for some positive constant~$c''$. An analogous computation shows that the same result holds also when~(B) is in force. By~\eqref{5.4} and choosing~$\varepsilon$ small enough, we compute
$$
E'(t) = B''(t) t - \varepsilon p^* t^{p^* - 1} \ge (c'' - \varepsilon p^*) t^{p^* - 1} \ge 0,
$$
which gives~\eqref{b.1}.
\end{proof}

Notice that, in the setting of the paper, Lemma~\ref{blem} actually comes as a corollary of Lemma~\ref{blem2}. Nevertheless, we preferred to state them independently one to the other, since the hypotheses required by the first do not involve the function~$H$ at all.

Finally, we present a lemma ensuring the continuity of the second derivative of~$B$ at the origin starting from some regularity assumptions on the composition~$B \circ H$. The framework in which this result is meant to be set is that of hypothesis~(B) and, in fact, explicit use of it will be made in Section~\ref{S:AUX2}.

\begin{lemma} \label{der2B0}
Let~$H \in C^2(\R^n \setminus \{ 0 \})$ be a positive homogeneous function of degree~$1$ satisfying~\eqref{Hsign} and~$B \in C^1([0, +\infty)) \cap C^2((0, +\infty))$, with~$B(0) = B'(0) = 0$. Assume in addition that~$B \circ H$ has some pure second derivative, say, the first, continuous at the origin. Then,~$B \in C^2([0, +\infty))$ with
\bequ \label{der2B0th}
B''(0) = H^{-2}(e_1) \frac{\partial^2 (B \circ H)}{\partial \xi_1^2} \, (0).
\equ
In particular, this holds if~$B \circ H \in C^2(\R^n)$.
\end{lemma}

\begin{proof}
Since, for every $\xi \neq 0$,
$$
\frac{\partial^2 (B \circ H)}{\partial \xi_1^2}(\xi) = B''(H(\xi)) H_1^2(\xi) + B'(H(\xi)) H_{1 1}(\xi),
$$
by choosing~$\xi = t e_1$, with~$t > 0$, and the homogeneity properties of~$H$ we obtain
\bequ \label{der2B0i}
\frac{\partial^2 (B \circ H)}{\partial \xi_1^2} \, (t e_1) = B''(t H(e_1)) H_1^2(e_1) + \frac{B'(t H(e_1))}{t} \, H_{1 1}(e_1).
\equ
Now, observe that
$$
H_1(e_1) = \nabla H(e_1) \cdot e_1 = H(e_1) > 0,
$$
by~\eqref{i} and
$$
H_{1 1}(e_1) = \nabla H_1(e_1) \cdot e_1 = 0,
$$
by~\eqref{ii}. Therefore, by~\eqref{der2B0i} we get
$$
B''(t H(e_1)) =  H^{-2}(e_1) \frac{\partial^2 (B \circ H)}{\partial \xi_1^2} \, (t e_1),
$$
which yields~\eqref{der2B0th} by passing to the limit as~$t \rightarrow 0^+$.
\end{proof}

\section{Regularity of the solutions}\label{S:2}

In this short section we point out some regularity properties of the weak solutions of~\eqref{eleq}.

\begin{proposition} \label{regprop1}
Let~$u$ be as in Theorem~\ref{gradpest}. Then, given any~$x_0 \in \R^n$ and~$R \in (0, 1)$, there exist~$\alpha \in (0,1)$ and~$C>0$, 
depending only on~$n$, $R$, $\| u \|_{L^\infty(\R^n)}$ and the constants involved in~(A) or~(B), so that
\begin{eqnarray}
& \| \nabla u \|_{L^\infty(\R^n)} \le C, \label{gradest} \\
& |\nabla u(x) - \nabla u(y)| \le C R^{- \alpha } {|x - y|}^\alpha, \label{gradhold}
\end{eqnarray}
for any $x, y \in B_R(x_0)$. In particular,~$u \in C_\loc^{1, \alpha}(\R^n)$, for such~$\alpha$.
\end{proposition}

\begin{proof}
In case~$(i)$ of Theorem~\ref{gradpest} holds, we can apply Theorem~1 on page~127 of~\cite{T84}. Notice that the ellipticity and growth conditions there required are satisfied by assumption (A) and the structure of equation~\eqref{eleq}. Condition~(1.7) of~\cite{T84} is also valid, due to the fact that~$f$ is continuous and~$u$ bounded. Finally, the locally boundedness of the gradient
in~\cite{T84}
could be easily extended to the whole of~$\R^n$, giving 
\eqref{gradest}. See also \cite{DiB83}.\\
If on the other hand~$(ii)$ is in force, then~\eqref{gradest} is already satisfied. In order to obtain~\eqref{gradhold}, the uniform ellipticity of the Hessian of~$B(H(\nabla u))$ allows us to appeal to Theorem~1.1 on page~339 of~\cite{LU68}
(notice that we know in addition that~$u\in W_\loc^{2, 2}(\R^n)$ in this
case, thanks to Proposition~1 in~\cite{T84},
the boundedness of~$\nabla u$ and the structural conditions in~$(ii)$).
\end{proof}

If we stay far from the points on which~$\nabla u$ vanishes, then we can obtain even more regularity for~$u$, as displayed by the following result:

\begin{proposition} \label{regprop2}
Let~$u$ be as in Theorem~\ref{gradpest}. Then, for
any~$x \in \R^n$ with~$\nabla u(x) \ne 0$
there exists~$R > 0$ and~$\alpha \in (0, 1)$ such that~$u
\in C^{3, \alpha}(B_R(x))$. 

In particular, we have that~$u \in C^{3}(\{ \nabla u \ne 0 \})$.

Moreover, if assumption~(ii) in
Theorem~\ref{gradpest} holds, we have the
stronger conclusion that~$u \in C^{3,\alpha}_\loc(\R^n)$.
\end{proposition}

\begin{proof}
If~$u$ satisfies~$(i)$ of Theorem~\ref{gradpest} and~$x$
is as in the statement, then we may apply Theorem~6.4 on
page~284 of~\cite{LU68} in some neighborhood of~$x$ contained
in~$\{ \nabla u \ne 0 \}$, which exists due to the
continuity of~$\nabla u$ granted by Proposition~\ref{regprop1},
to obtain the thesis.

The same result also holds if condition~$(ii)$ is valid,
relying instead on Theorem~6.3
on page~283 of~\cite{LU68}. Note that, in this case,
the non-degeneracy of~$\nabla u$ is no longer required,
obtaining that~$u$ is actually of class~$C^{3,\alpha}_\loc$ on the whole of~$\R^n$.
\end{proof}

\section{$P$-function computations}\label{S:3}

Now we perform a $P$-function argument, by showing
that a suitable energy functional is a subsolution
of a partial differential equation (in fact, it is a solution, with a remainder term which has a sign).
Classical computations of this kind are in~\cite{P76, S81}.

For the sake of briefness, in the following we will often adopt the notation~$H = H(\nabla u)$, $H_i = (\partial_i H) (\nabla u)$, $B= B(H(\nabla u))$, $B'= B'(H(\nabla u))$, etc.

\begin{proposition}\label{SP:SP}
Let~$u$ be as in Theorem~\ref{gradpest}. Set
\begin{eqnarray}
& G(r) := c_u - F(r) \mbox{ for any } r \in \R, \\ \label{Gdef}
& a_{i j} := B'' H_i H_j + B' H_{i j}, \qquad \qquad d_{i j} := a_{i j} / H,  \label{addef}
\end{eqnarray}
and
\bequ \label{Pdef}
P(u; x) := B'(H(\nabla u(x))) H(\nabla u(x)) - B(H(\nabla u(x))) - G(u(x)).
\equ
Then,
\bequ \label{Pine}
{(d_{i j} P_i)}_j - b_k P_k ={\mathscr{R}}
\ge 0 \mbox{ on } \{ \nabla u 
\ne 0 \},
\equ
where
\begin{equation}\label{9090}\begin{split}
&b_k := \frac{B'''}{B''} H^{-2} H_\ell P_\ell H_k + \left[ 
\frac{B'''}{B''} 
+ \frac{B''}{B'} \right] G' H^{-1} H_k + \left[ \frac{B' B'''}{{(B'')}^2} 
+ 1 \right] H^{-2} H_{k \ell} P_\ell\\
{\mbox{and }}\ &{\mathscr{R}}:=
B' B'' H_{i j} H_{k \ell} u_{i k} u_{j \ell}.\end{split}\end{equation}
\end{proposition}

\begin{proof}
First of all, we point out that, by Proposition~\ref{regprop2},~$u$ is~$C^3(\{\nabla u \ne 0\})$. We will therefore implicitly assume every calculation to be performed on~$\{\nabla u \ne 0\}$. 
The computation is quite long and somehow delicate, but we
provide full details of the argument for the facility of the reader.
By differentiating~\eqref{Pdef}, we get for any $i\in\{1,\dots,n\}$
\bequ \label{Pi}
P_i = B'' H H_k u_{k i} + B' H_k u_{k i} - B' H_k u_{k i} - G' u_i = B'' H H_k u_{k i} - G' u_i.
\equ
Thus, recalling~\eqref{addef},
\bequ \label{form1}
\begin{split}
{(d_{i j} P_i)}_j & = {(B'' H H_k d_{i j} u_{k i})}_j - {(G' d_{i j} u_i)}_j \\
& = {(B'' H_k)}_j a_{i j} u_{k i} + B'' H_k {(a_{i j} u_{k i})}_j - {(G' d_{i j} u_i)}_j.
\end{split}
\equ
Next, observe that from~\eqref{eleq} we have
\bequ \label{eleq2}
a_{i j} u_{i j} = G'.
\equ
Being~$u$ of class~$C^3$, we compute for any~$k$
\bequ
\begin{split}
& {(a_{i j} u_{k i} )}_j - {(a_{i j} u_{i j} )}_k = {(a_{i j})}_j u_{k i} - {(a_{i j})}_k u_{i j} \\
& \qquad = \left[ B''' H_i H_j H_\ell + B'' H_{i \ell} H_j + B'' H_i H_{j \ell} + B'' H_{i j} H_\ell + B' H_{i j \ell} \right] u_{j \ell} u_{k i} \\
& \qquad - \, \left[ B''' H_i H_j H_\ell + B'' H_{i \ell} H_j + B'' H_i H_{j \ell} + B'' H_{i j} H_\ell + B' H_{i j \ell} \right] u_{k \ell} u_{i j} = 0,
\end{split}
\equ
by interchanging the indices~$i$ and~$\ell$ in the last term. Therefore, using~\eqref{eleq2} we obtain
$$
{(a_{i j} u_{k i} )}_j = {(a_{i j} u_{i j} )}_k = {(G')}_k = G'' u_k.
$$
Plugging this into~\eqref{form1} we have
\bequ \label{form2}
\begin{split}
{(d_{i j} P_i)}_j & = {(B'' H_k)}_j a_{i j} u_{k i} + B'' H_k G'' u_k - {(G' d_{i j} u_i)}_j \\
& = {(B'' H_k)}_j a_{i j} u_{k i} + B'' H_k G'' u_k - G'' d_{i j} u_i u_j - G' {(d_{i j} u_i)}_j.
\end{split}
\equ
Now, we collect the two terms containing~$G''$, getting, by~\eqref{i} and~\eqref{ii},
\begin{equation*}
\begin{split}
B'' H_k G'' u_k - G'' d_{i j} u_i u_j & = G'' H^{-1} \left[ B'' H H_k u_k - a_{i j} u_i u_j \right] \\
& = G'' H^{-1} \left[ B'' H^2 - B'' H_i H_j u_i u_j - B' H_{i j} u_i u_j \right] \\
& = G'' H^{-1} \left[ B'' H^2 - B'' H^2 - 0 \right] = 0.
\end{split}
\end{equation*}
Hence,~\eqref{form2} becomes
\bequ \label{form3}
\begin{split}
{(d_{i j} P_i)}_j & = {(B'' H_k)}_j a_{i j} u_{k i} - G' {(d_{i j} u_i)}_j \\
& = {(B'' H_k)}_j a_{i j} u_{k i} - G' {(d_{i j})}_j u_i - G' d_{i j} u_{i j} \\
& = {(B'' H_k)}_j a_{i j} u_{k i} - G' {(d_{i j})}_j u_i - {(G')}^2 H^{-1},
\end{split}
\equ
where in the last line we made use of~\eqref{eleq2}. Appealing to~\eqref{i},~\eqref{ii} and~\eqref{iii}, we compute
\bequ \label{form4}
\begin{split}
{(d_{i j})}_j u_i & = {\left[ B'' H^{-1} H_i H_j + B' H^{-1} H_{i j} \right]}_j u_i \\
& = \left[ B''' H^{-1} H_i H_j H_\ell - B'' H^{-2} H_i H_j H_\ell + B'' H^{-1} H_{i \ell} H_j + B'' H^{-1} H_i H_{j \ell} \right. \\
& \quad \left. + B'' H^{-1} H_{i j} H_\ell - B' H^{-2} H_{i j} H_\ell + B' H^{-1} H_{i j \ell} \right] u_{j \ell} u_i \\
& = \left[ B''' H_j H_\ell - B'' H^{-1} H_j H_\ell + 0 + B'' H_{j \ell} + 0 - 0 - B' H^{-1} H_{j \ell} \right] u_{j \ell} \\
& = \left[ B''' - B'' H^{-1} \right] H_j H_\ell u_{j \ell} + \left[ B'' - B' H^{-1} \right] H_{j \ell} u_{j \ell}.
\end{split}
\equ
Writing explicitly~\eqref{eleq2}
$$
G' = a_{i j} u_{i j} = B'' H_i H_j u_{i j} + B' H_{i j} u_{i j},
$$
we deduce
\bequ \label{eleq2inv}
H_{i j} u_{i j} = {(B')}^{-1} \left[ G' - B'' H_i H_j u_{i j} \right].
\equ
By this equation,~\eqref{form4} becomes
\bequ \label{form5}
\begin{split}
{(d_{i j})}_j u_i & = \left[ B''' - B'' H^{-1} \right] H_j H_\ell u_{j \ell} + {(B')}^{-1} \left[ B'' - B' H^{-1} \right] \left[ G' - B'' H_j H_\ell u_{j \ell} \right] \\
& = \left[ B''' - {(B')}^{-1} {(B'')}^2 \right] H_j H_\ell u_{j \ell} + G' {(B')}^{-1} \left[ B'' - B' H^{-1} \right] .
\end{split}
\equ
Now, inverting~\eqref{Pi}, we get
\bequ \label{Piinv}
H_k u_{k i} = {( B'' H )}^{-1} \left[ P_i + G' u_i \right].
\equ
Exploiting~\eqref{Piinv} in~\eqref{form5} and using~\eqref{i}, we obtain
\begin{equation*}
\begin{split}
{(d_{i j})}_j u_i & = {( B'' H )}^{-1} \left[ B''' - {(B')}^{-1} {(B'')}^2 \right] \left[ P_\ell + G' u_\ell \right] H_\ell + G' {(B')}^{-1} \left[ B'' - B' H^{-1} \right] \\
& = H^{-1} \left[ {( B'' )}^{-1} B''' - {(B')}^{-1} B'' \right] \left[ P_\ell + G' u_\ell \right] H_\ell + G' {(B')}^{-1} \left[ B'' - B' H^{-1} \right] \\
& = H^{-1} \left[ {( B'' )}^{-1} B''' - {(B')}^{-1} B'' \right]  H_\ell P_\ell + G' \left[ {( B'' )}^{-1} B''' - H^{-1} \right].
\end{split}
\end{equation*}
By this last equality,~\eqref{form3} becomes
\begin{align}
\nonumber
{(d_{i j} P_i)}_j & = {(B'' H_k)}_j a_{i j} u_{k i} - G' H^{-1} \left[ {( B'' )}^{-1} B''' - {(B')}^{-1} B'' \right]  H_\ell P_\ell \\ \label{form6}
& \quad - {(G')}^2 \left[ {( B'' )}^{-1} B''' - H^{-1} \right] - {(G')}^2 H^{-1} \\ \nonumber
& = {(B'' H_k)}_j a_{i j} u_{k i} - G' H^{-1} \Big[ {( B'' )}^{-1} B''' - {(B')}^{-1} B'' \Big]  H_\ell P_\ell - {(G')}^2 {( B'' )}^{-1} B'''.
\end{align}
Now, we use~\eqref{Piinv} to write, for any~$j$ and~$k$,
\begin{equation*}
\begin{split}
{(B'' H_k)}_j & = B''' H_k H_\ell u_{j \ell} + B'' H_{k \ell} u_{j \ell} \\
& = B''' H_k {( B'' H )}^{-1} \left[ P_j + G' u_j \right] + B'' H_{k \ell} u_{j \ell} \\
& = {( B'' )}^{-1} B''' H^{-1} H_k P_j + G' {( B'' )}^{-1} B''' H^{-1} H_k u_j  + B'' H_{k \ell} u_{j \ell},
\end{split}
\end{equation*}
and
\begin{equation*}
\begin{split}
a_{i j} u_{i k} & = \left[ B'' H_i H_j + B' H_{i j} \right] u_{i k} = B'' H_i H_j u_{i k} + B' H_{i j} u_{i k} \\
& = B'' H_j {( B'' H )}^{-1} \left[ P_k + G' u_k \right] + B' H_{i j} u_{i k} \\
& = H^{-1} H_j P_k + G' H^{-1} H_j u_k + B' H_{i j} u_{i k}.
\end{split}
\end{equation*}
We put together the two formulae just obtained, getting
\begin{equation*}
\begin{split}
{(B'' H_k)}_j a_{i j} u_{k i} & = \left[ {( B'' )}^{-1} B''' H^{-1} H_k P_j + G' {( B'' )}^{-1} B''' H^{-1} H_k u_j  + B'' H_{k \ell} u_{j \ell} \right] \\
& \quad \times \left[ H^{-1} H_j P_k + G' H^{-1} H_j u_k + B' H_{i j} u_{i k} \right] \\
& = {( B'' )}^{-1} B''' H^{-2} H_k P_j H_j P_k + G' {( B'' )}^{-1} B''' H^{-2} H_k P_j H_j u_k \\
& \quad + B' {( B'' )}^{-1} B''' H^{-1} H_k P_j H_{i j} u_{i k} + G' {( B'' )}^{-1} B''' H^{-2} H_k u_j H_j P_k \\
& \quad + {(G')}^2 {( B'' )}^{-1} B''' H^{-2} H_k u_j H_j u_k + G' B' {( B'' )}^{-1} B''' H^{-1} H_k u_j H_{i j} u_{i k} \\
& \quad + B'' H^{-1} H_{k \ell} u_{j \ell} H_j P_k + G' B'' H^{-1} H_{k \ell} u_{j \ell} H_j u_k + B' B'' H_{k \ell} u_{j \ell} H_{i j} u_{i k}.
\end{split}
\end{equation*}
Making use of~\eqref{Piinv},~\eqref{i} and~\eqref{ii}, this becomes
\begin{equation*}
\begin{split}
{(B'' H_k)}_j a_{i j} u_{k i} & = {( B'' )}^{-1} B''' H^{-2} {(H_\ell P_\ell)}^2 + G' {( B'' )}^{-1} B''' H^{-1} H_\ell P_\ell \\
& \quad + B' {( B'' )}^{-1} B''' H^{-1} P_j H_{i j} {( B'' H )}^{-1} \left[ P_i + G' u_i \right] + G' {( B'' )}^{-1} B''' H^{-1} H_\ell P_\ell \\
& \quad + {(G')}^2 {( B'' )}^{-1} B''' + 0 \\
& \quad + B'' H^{-1} H_{k \ell} P_k {( B'' H )}^{-1} \left[ P_\ell + G' u_\ell \right] + 0 + B' B'' H_{i j} H_{k \ell} u_{i k} u_{j \ell}.
\end{split}
\end{equation*}
Developing the products and exploiting again~\eqref{ii}, we have
\begin{equation*}
\begin{split}
{(B'' H_k)}_j a_{i j} u_{k i} & = {( B'' )}^{-1} B''' H^{-2} {(H_\ell P_\ell)}^2 + G' {( B'' )}^{-1} B''' H^{-1} H_\ell P_\ell \\
& \quad + B' {( B'' )}^{-2} B''' H^{-2} H_{i j} P_i P_j + 0 + G' {( B'' )}^{-1} B''' H^{-1} H_\ell P_\ell \\
& \quad + {(G')}^2 {( B'' )}^{-1} B''' + H^{-2} H_{k \ell} P_k  P_\ell + 0 + B' B'' H_{i j} H_{k \ell} u_{i k} u_{j \ell}.
\end{split}
\end{equation*}
Simplifying and collecting similar terms, we get
\begin{equation*}
\begin{split}
{(B'' H_k)}_j a_{i j} u_{k i} & = {( B'' )}^{-1} B''' H^{-2} {(H_\ell P_\ell)}^2 + 2 G' {( B'' )}^{-1} B''' H^{-1} H_\ell P_\ell + B' B'' H_{i j} H_{k \ell} u_{i k} u_{j \ell} \\
& \quad + H^{-2} \left[ B' {( B'' )}^{-2} B''' + 1 \right] H_{k \ell} P_k  P_\ell + {(G')}^2 {( B'' )}^{-1} B'''.
\end{split}
\end{equation*}
Plugging this into~\eqref{form6} we finally obtain
\begin{equation*}
\begin{split}
{(d_{i j} P_i)}_j & = {( B'' )}^{-1} B''' H^{-2} {(H_\ell P_\ell)}^2 + 2 G' {( B'' )}^{-1} B''' H^{-1} H_\ell P_\ell + B' B'' H_{i j} H_{k \ell} u_{i k} u_{j \ell} \\
& \quad + H^{-2} \left[ B' {( B'' )}^{-2} B''' + 1 \right] H_{k \ell} P_k  P_\ell + {(G')}^2 {( B'' )}^{-1} B''' \\
& \quad - G' H^{-1} \left[ {( B'' )}^{-1} B''' - {(B')}^{-1} B'' \right]  H_\ell P_\ell - {(G')}^2 {( B'' )}^{-1} B''' \\
& = {( B'' )}^{-1} B''' H^{-2} {(H_\ell P_\ell)}^2 + G' H^{-1} \left[ {( B'' )}^{-1} B''' + {(B')}^{-1} B'' \right]  H_\ell P_\ell \\
& \quad + H^{-2} \left[ B' {( B'' )}^{-2} B''' + 1 \right] H_{k \ell} P_k  P_\ell + B' B'' H_{i j} H_{k \ell} u_{i k} u_{j \ell}.
\end{split}
\end{equation*}
The last term of the formula above coincides with the 
remainder~${\mathscr{R}}$
as defined in~\eqref{9090} and it
is non-negative by \eqref{Bsign} and 
via an application of Proposition~\ref{corpos} with~$c_{i j} := u_{i j}$.
Therefore, inequality~\eqref{Pine} is proved.
\end{proof}

\section{Proof of Theorem \ref{gradpest}}\label{S:4}

The proof is similar to the one of Theorem~1 in~\cite{FV13},
with the following modification. Formula~(32) of~\cite{FV13}
is replaced here with~\eqref{Pine}, on which we apply
the classical Maximum Principle.

\section{Proof of Theorem \ref{1D}}\label{S:4.4}

Up to a rotation and a translation, we may consider the origin
lying in a level set~$\{ u=c\}$,
with
\begin{equation}\label{7777}
\nabla u(0)=|\nabla u(0)|\,(0,\dots,0,1)\ne0.\end{equation}
We stress that the equation is not invariant under a rotation~$R$,
but the function~$H$ would be replaced by~$\tilde H:=H\circ R$.
Nevertheless, the new function~$\tilde H$
satisfies the same structural assumptions
of~$H$, thus we take the freedom of identifying~$\tilde H$ with
the original~$H$.

We parameterize the level set of~$u$ near the origin
with the graph of a $C^2$ function~$\phi$, i.e. we 
write~$u(x',\phi(x'))=c$
for~$x'\in\R^{n-1}$ near the origin.
By taking two derivatives, we obtain that
\begin{eqnarray*}
&& u_{i'}+u_n\phi_{i'}=0\\
{\mbox{and }}&&u_{i'j'}+u_{i'n}\phi_{j'}
+u_{j'n}\phi_{i'}+u_{nn}\phi_{i'}\phi_{j'}+u_n\phi_{i'j'}=0,
\end{eqnarray*}
for any~$i',j'\in\{1,\dots,n-1\}$,
where the derivatives of~$u$ are evaluated at~$(x',\phi(x'))$
and the derivatives of~$\phi$ are evaluated at~$x'$.
In particular, by taking~$x'=0$, we obtain that~$\phi_{i'}(0)=0$
and~$u_{i'j'}(0)=-u_{n}(0)\phi_{i'j'}(0)$,
for any~$i',j'\in\{1,\dots,n-1\}$.

Consequently, we have that all the 
principal curvatures of the level set at~$0$ vanish if and only if~$
\phi_{i'j'}(0)=0$ for any~$i',j'\in\{1,\dots,n-1\}$, and so,
by~\eqref{7777},
if and only if
\begin{equation}\label{8888}
{\mbox{$u_{i'j'}(0)=0$ for any~$i',j'\in\{1,\dots,n-1\}$.}}
\end{equation}
Hence, we establish~\eqref{8888} in order to complete the
proof of Theorem~\ref{1D}. The proof of~\eqref{8888}
is based on Proposition~\ref{corpos}.
We need to check that the hypotheses of Proposition~\ref{corpos}
are satisfied in this case.
First of all, we have that~\eqref{EQ1} is guaranteed by~\eqref{7777}
(here~$\xi=\nabla u(0)$). Then, by Theorem~\ref{gradpest},
we know that~\eqref{gradpideq2} holds true in the whole
connected component~${\mathscr{S}}$ that contains~$0$.
As a consequence, $P$ vanishes identically in~${\mathscr{S}}$,
thus we obtain from~\eqref{Pine} and~\eqref{9090} that
$$ 0=
{(d_{i j} P_i)}_j - b_k P_k ={\mathscr{R}}
= B' B'' H_{i j} H_{k \ell} u_{i k} u_{j \ell}.$$
This says
that equality holds in~\eqref{EQ0} with~$\xi=\nabla u(0)$
and~$c_{ij}=u_{ij}$. Accordingly, the hypotheses
of Proposition~\ref{corpos} are fulfilled and we obtain~\eqref{8888}
from~\eqref{2.14a}.
The proof of Theorem~\ref{1D}
is therefore complete.

\section{Proof of Theorem \ref{1DAUX2}}\label{S:AUX2}

In this case $u$ is of class $C^3$ everywhere, due to
Proposition~\ref{regprop2},
therefore
we can differentiate~\eqref{eleq} and write it
in non-divergence form as
$$ (B\circ H)_{ij} u_{ij}+F'(u)=0.$$
Notice that the matrix~$\{ (B\circ H)_{ij} \}_{i,j\in\{1,\dots,n\} }$
is elliptic since, by assumption~(B),
\begin{equation} \label{ELL1}
(B\circ H)_{ij}\xi_i\xi_j \ge \gamma|\xi|^2.
\end{equation}
Moreover, by Lemma~\ref{der2B0},~$B$ is of class~$C^2$ at the origin with
\begin{equation} \label{ELL2}
B''(0) > 0.
\end{equation}
Notice in particular that the last inequality follows combining~\eqref{der2B0th} and~\eqref{ELL1}.

Now, we observe that, in view of Theorem~\ref{1D},~$u$ is one-dimensional
and that its profile~$u_0$ satisfies the ordinary
differential equation~\eqref{ODE} on an interval. Also recall that~$u$, and consequently~$u_0$, is bounded, with bounded gradient.

Thanks to~\eqref{ELL2},
the linearized equation can be represented as a first order system of ODEs in canonical form,
so that~$u_0$ extends to a global solution~$\hat{u}_0$
by the standard theory for Cauchy problems
with globally Lipschitz nonlinearities (see e.g. page~146 in~\cite{PSV84}).

Finally, by the Unique Continuation Principle (see e.g.~\cite{H05}), we have
that~$u$
agrees everywhere with the one-dimensional extension of~$\hat{u}_0$ to~$\R^n$.
This concludes the proof.

\section{Proof of Theorem \ref{1DAUX}}\label{S:AUX}

Assume~$\S$ to be any connected component of~$\{ \nabla u \ne 0 \}$. We claim that
\begin{equation} \label{4.1}
\begin{split}
& \mbox{$\S$ is foliated by level sets of $u$ which are union of parallel affine hyperplanes} \\
& \mbox{and so $\partial \S$ is the union of (at most two) parallel hyperplanes.} 
\end{split}
\end{equation}
In order to prove this, fix~$x_\star \in \S$ and consider the level set~$S_{x_\star} := \{ u = u(x_\star) \}$. Notice that
\bequ \label{4.2}
S_{x_\star} \subseteq \{ \nabla u \ne 0 \}.
\equ
Indeed, if~$x \in S_{x_\star}$, then from~\eqref{gradpideq2} we deduce that
\begin{multline} \label{4.3}
B'(H(\nabla u(x))) H(\nabla u(x)) - B(H(\nabla u(x))) = c_u - F(u(x)) \\
= c_u - F(u(x_\star)) = B'(H(\nabla u(x_\star))) H(\nabla u(x_\star)) - B(H(\nabla u(x_\star))) > 0,
\end{multline}
because~$x_\star$ is in~$\{ \nabla u \ne 0 \}$ and so we can apply~\eqref{bpos}
taking~$t := H(\nabla u(x_\star)) > 0$. But then, also~$x \in \{\nabla u \ne 0\}$,
since otherwise~$H(\nabla u)$ would vanish, in contradiction with~\eqref{4.3}.
This establishes~\eqref{4.2} so that we are allowed to apply Theorem~\ref{1D},
concluding that every connected component of~$S_{x_\star}$ is contained in a hyperplane, say~$\ell_{x_\star}$. In particular, we point out that
\bequ \label{4.4}
\mbox{the connected component of~$S_{x_\star}$ which contains~$x_\star$ is equal to~$\ell_{x_\star}$.}
\equ
Indeed,~$S_{x_\star}$ is closed in the relative topology of~$\ell_{x_\star}$, being $u$ continuous.
Furthermore,~$S_{x_\star}$ is also relatively open, by~\eqref{4.2} and applying Theorem~\ref{1D}
together with the Implicit Function Theorem. Thus,~\eqref{4.4} holds true.

Combining~\eqref{4.4} and~\eqref{4.2} we immediately obtain~\eqref{4.1}.

Let now~$\omega$ denote a vector normal to all the hyperplanes in~\eqref{4.1}. We claim that
\bequ \label{4.5}
u(x_0) = u(y_0) \mbox{ if } (x_0 - y_0) \cdot \omega = 0.
\equ
To check this, fix~$x_0 \in \R^n$. If~$\nabla u = 0$ on
the whole~$\ell_{x_0}$, then~\eqref{4.5} follows from the Fundamental Theorem of Calculus.
Conversely, let~$x_\sharp$ be a point in~$\ell_{x_0} \cap \{ \nabla u \ne 0 \}$.
By~\eqref{4.1} (applied to~$x_\sharp$), we have that~$u$ is constant on~$\ell_{x_\sharp}$, which,
in turn, is equal to~$\ell_{x_0}$. Thus,~\eqref{4.5} is proved and so is the desired one-dimensional Euclidean symmetry.

\section{Proof of Theorem \ref{liou}}\label{S:7}

Let~$r := u(x_0)$ and fix a point~$x \in \R^n \setminus \{ x_0 \}$. In order to establish the thesis of Theorem~\ref{liou}, we shall show that~$u(x) = r$. Consider the~$C^1$ function~$\varphi$, defined by setting
$$
\varphi(t) := u(t x + (1 - t) x_0) - r \mbox{ for any } t \in [0, 1].
$$
In the following we will sometimes adopt the short notation~$x_t := t x + (1 - t) x_0$. Notice that, by the regularity of~$u$, the function~$t \mapsto |\nabla u(t x + (1 - t) x_0)|$ is bounded on~$[0, 1]$. We may therefore apply Lemma~\ref{blem2} (and also recall the notation there introduced) to compute
\begin{equation*}
\begin{split}
{|\dot{\varphi}(t)|}^{p^*} & \le {|x - x_0|}^{p^*} {|\nabla u(t x + (1 - t) x_0)|}^{p^*} \\
& = {|x - x_0|}^{p^*} {|\nabla u(x_t)|}^{p^*} \\
& \le \frac{{|x - x_0|}^{p^*}}{h^{p^*}} \, H^{p^*} \big( \nabla u(x_t) \big) \\
& \le \frac{{|x - x_0|}^{p^*}}{\varepsilon h^{p^*}} \, \left[ B'(H(\nabla u(x_t)) H(\nabla u(x_t)) - B(H(\nabla u(x_t))) \right],
\end{split}
\end{equation*}
for some~$\varepsilon > 0$. Next, recalling~\eqref{gradpesteq} and the assumptions of Theorem~\ref{liou}, we have that
\begin{equation} \label{5.1}
\begin{split}
{|\dot{\varphi}(t)|}^{p^*} & \le \frac{{|x - x_0|}^{p^*}}{\varepsilon h^{p^*}} \, \left[ c_u - F(u(x_t)) \right] \\
& = \frac{{|x - x_0|}^{p^*}}{\varepsilon h^{p^*}} \, \left[ F(r) - F(u(t x + (1 - t) x_0)) \right] \\
& = - \frac{{|x - x_0|}^{p^*}}{\varepsilon h^{p^*}} \int_r^{u(t x + (1 - t) x_0)} F'(\sigma) \, d\sigma.
\end{split}
\end{equation}
Then, we employ alternatively the Lipschitz regularity of~$F'$ or~\eqref{Fgrowliou} to write
$$
|F'(\sigma)| \le c \, {|r - \sigma|}^{p^* - 1} \mbox{ for any } \sigma \in \left[ \inf_{\R^n} u, \sup_{\R^n} u \right],
$$
for some positive constant~$c$. Using this estimate in~\eqref{5.1}, we get
\begin{equation*}
\begin{split}
{|\dot{\varphi}(t)|}^{p^*} & \le \frac{c {|x - x_0|}^{p^*}}{\varepsilon h^{p^*}} \left| \int_r^{u(t x + (1 - t) x_0)} {|r - \sigma|}^{p^* - 1} d\sigma \right| \\
& = \frac{c {|x - x_0|}^{p^*}}{\varepsilon p^* h^{p^*}} \, {|u(t x + (1 - t) x_0) - \sigma|}^{p^*} \\
& = \frac{c {|x - x_0|}^{p^*}}{\varepsilon p^* h^{p^*}} \, {|\varphi(t)|}^{p^*},
\end{split}
\end{equation*}
which yields, if~$\varphi(t) \ne 0$,
$$
\left| \frac{\dot{\varphi}(t)}{\varphi(t)} \right| \le \frac{c^{1/p^*} |x - x_0|}{\varepsilon {p^*}^{1/p^*} h} =: K.
$$
Finally, set~$\psi(t) := {(\varphi(t))}^2 e^{-K t}$, for any~$t \in [0, 1]$. Differentiating~$\psi$, we obtain
\begin{equation*}
\begin{split}
\dot{\psi}(t) & = \varphi(t) e^{-K t} \left[ \dot{\varphi}(t) - K \varphi(t) \right] \\
& =
\left\{
\begin{array}{cc}
{(\varphi(t))}^2 e^{-K t} \left[ \frac{\dot{\varphi}(t)}{\varphi(t)} - K \right] & \mbox{ if } \varphi(t) \ne 0 \\
0 & \mbox{ if } \varphi(t) = 0
\end{array}
\right. \\
& \le 0,
\end{split}
\end{equation*}
so that~$\psi$ is non-increasing. Hence
$$
{(u(x) - r)}^2 e^{-K} = {\varphi(1)}^2 e^{-K} = \psi(1) \le \psi(0) = {\varphi(0)}^2 = {(u(x_0) - r)}^2 = 0,
$$
and therefore~$u(x) = r$, which concludes the proof.

\section{Proof of Theorem \ref{liou2}}\label{S:8}

We will suppose, without loss of generality, that
\bequ \label{uconst}
\mbox{$u$ is not constant.}
\equ
Then, assume by contradiction that there exists~$r_0 \in \big( \inf_{\R^n} u, \sup_{\R^n} u \big)$ such that
$$
\sup \left\{ F(r) : r \in \left[ \inf_{\R^n} u, \sup_{\R^n} u \right] \right\} = c_u = F(r_0).
$$
By the continuity of~$u$, there also exists a point~$x_0 \in \R^n$ such that~$u(x_0) = r_0$. Moreover,~$r_0$ is a local maximum for~$F$, so that~$F'(r_0) = 0$. Thus, we can apply Theorem~$\ref{liou}$, deducing that~$u$ is constant, in contradiction to~\eqref{uconst}.

\appendix
\section{An example in which $H$ is not a norm} \label{notanorm}

Here we present an example in which~$H$ satisfies
the structural assumptions requested in this paper without
being a norm. Indeed~$H$ will be positive homogeneous of degree~$1$,
but not necessarily a norm. More precisely,
given any convex set~$\K$ described as a graph over the sphere by
$$
\K:= \left\{ \frac{t \xi}{\Theta(\xi)}, \ \xi\in S^{n-1}, \ t\in[0,1]
\right\},
$$
with~$\Theta\in C^{3,\beta}_\loc(\R^n\setminus\{0\},\,(0,+\infty))\cap L^\infty_\loc(\R^n)$, and
with 
\begin{equation}\label{p.9}
{\mbox{principal curvatures along $\partial \K$
bounded from below by some $c>0$,}}\end{equation}
we construct an admissible $H$ for which
\begin{equation}\label{H1}
\{ H=1\}=\partial \K.
\end{equation}
Precisely, such~$H$ is defined,
for any~$\xi\in\R^n\setminus\{0\}$, by
\bequ \label{HHdef}
H(\xi):= |\xi|\,\Theta\left( \frac{\xi}{|\xi|}\right).
\equ
Notice that $H$ is not even, unless so is~$\Theta$.
Therefore, in general, $H$ is not a norm. We have that
\begin{align*}
\partial{\K} & = \left\{\zeta\in\R^n : \zeta=\frac{\xi}{\Theta(\xi)}, \, \xi\in S^{n-1} \right\} \\
& = \left\{\zeta\in\R^n :|\zeta|=\frac{1}{\Theta(\xi)}, \, \xi=\frac{\zeta}{|\zeta|}\in S^{n-1} \right\} \\
& = \left\{\zeta\in\R^n : H(\zeta)=1 \right\},
\end{align*}
that is~\eqref{H1}.

Our goal is to show that
\begin{equation} \label{scope}
{\mbox{if }} B(t)=\frac{t^p}{p} {\mbox{ with $p>1$, then assumption (A) holds.}}
\end{equation}
To this end, we first notice that
\begin{equation}\label{poHH}
{\mbox{$H$ is positively homogeneous of degree $1$,}}\end{equation} and so the range of~$\Hess \, (H) (\xi)$
lies in~$\xi^\perp$, thanks to point~$(ii)$ in Lemma~\ref{homain}.
Then we show that
\begin{equation} \label{END}
{\mbox{$\Hess \, (H) (\xi)$ is a positive definite endomorphism on~$\xi^\perp$.}}
\end{equation}
To see this, we make the relation between the second fundamental form of~$\partial \K$
and the Hessian of~$H$ explicit. Although we believe this fact to be well-known to the experts, we still
provide all the details. Let~$\xi \in \R^n \setminus \{ 0 \}$ with~$H(\xi) = 1$
and~$v, w \in T_\xi (\partial \K)$. By indicating with~$\nu = (\nu_1, \ldots, \nu_n)$ the interior
normal of~$\partial \K$ at~$\xi$, we obtain that
the second fundamental form of~$\partial \K$ at~$\xi$ applied to~$v$ and~$w$ equals
$$
\mathrm{I\!I}_\xi(v, w) = - \langle d_\xi \nu(v), w \rangle = - \frac{\partial \nu_j}{\partial \xi_i} v_i w_j.
$$
Since~$\nu (\xi)= - \nabla H(\xi) / |\nabla H(\xi)|$, we compute
$$
\frac{\partial \nu_j}{\partial \xi_i} = - \frac{H_{i j}(\xi)}{|\nabla H(\xi)|} +
\frac{H_i(\xi) H_j(\xi)}{|\nabla H(\xi)|^3}.
$$
Being by definition~$v, w \perp \nabla H(\xi)$, we obtain
\begin{equation}\label{8d89q}
\mathrm{I\!I}_\xi(v, w) = \frac{H_{i j}(\xi)}{|\nabla H(\xi)|} v_i w_j
- \frac{H_i(\xi) H_j(\xi)}{|\nabla H(\xi)|^3} v_i w_j =
\frac{H_{i j}(\xi) v_i w_j}{|\nabla H(\xi)|}.
\end{equation}
Furthermore, by~\eqref{HHdef}
$$
|\xi| = \frac{H(\xi)}{\Theta \left( \frac{\xi}{|\xi|} \right)} =
\frac{1}{\Theta \left( \frac{\xi}{|\xi|} \right)} \le {\left[ \min_{\zeta \in S^{n - 1}} \Theta(\zeta) \right]}^{-1} =: c',
$$
and so, using~\eqref{poHH}
$$
|\nabla H(\xi)| \ge \frac{1}{|\xi|} \, \nabla H(\xi) \cdot \xi = \frac{H(\xi)}{|\xi|} = \frac{1}{|\xi|} \ge \frac{1}{c'},
$$
Therefore, by~\eqref{p.9} and~\eqref{8d89q}
we conclude that, for any~$\xi\in\R^n\setminus\{0\}$ with~$H(\xi)=1$ and any~$v \in {\nabla H(\xi)}^\perp$,
\begin{equation}\label{d7f7}
H_{i j}(\xi) v_i v_j = |\nabla H(\xi)| \, \mathrm{I\!I}_\xi(v, v) \ge \frac{c}{c'} |v|^2.
\end{equation}
Now, by homogeneity we extend the previous estimate to
any~$\xi \in \R^n \setminus \{ 0 \}$.
For this, fixed any~$\xi\in\R^n\setminus\{0\}$ and any~$v \in {\nabla H(\xi)}^\perp$,
we define~$\tilde\xi:=\xi/H(\xi)$. By~\eqref{poHH}, we have that~$H(\tilde\xi)=1$
and that~$\nabla H(\tilde\xi)=\nabla H(\xi)$. Hence we can apply~\eqref{d7f7}
to~$\tilde\xi$ and~$v\in {\nabla H(\tilde\xi)}^\perp$, and once more the homogeneity in~\eqref{poHH}, obtaining
\begin{equation} \label{geomell.PRE}
H_{i j}(\xi) v_i v_j =H_{i j}\big(H(\xi) \tilde \xi\big) v_i v_j=
(H(\xi))^{-1} H_{i j}(\tilde\xi) v_i v_j
\ge \frac{c}{c'} (H(\xi))^{-1} |v|^2. 
\end{equation}
On the other hand,
$$ H(\xi)\in \Big[ |\xi| \min_{\zeta\in S^{n-1}}\Theta(\zeta),\ |\xi| \max_{\zeta\in S^{n-1}}\Theta(\zeta)\Big].$$
This and~\eqref{geomell.PRE} imply that
\begin{equation} \label{geomell}
H_{i j}(\xi) v_i v_j 
\ge \tilde{c} |\xi|^{-1} |v|^2,
\end{equation}
for some~$\tilde{c} > 0$.
We now complete the proof of~\eqref{END}. We observe that it is enough to prove that,
for any~$\xi, \eta \in S^{n - 1}$, with $\eta \in \xi^\perp$,
\begin{equation} \label{END2}
H_{ij}(\xi)\eta_i\eta_j\ge \tilde{c}.
\end{equation}
Since~$\xi$ is not orthogonal to~$\nabla H(\xi)$, we can write~$\eta =
\alpha \xi + v$, for some~$\alpha \in \R$ and~$v \perp \nabla H(\xi)$.
On the other hand, being~$\eta$ and~$\xi$ orthogonal, by Pythagoras' theorem we have that
\begin{equation} \label{vest}
|v|^2 = |\eta|^2 + \alpha^2 |\xi|^2 = 1 + \alpha^2 \ge 1.
\end{equation}
Therefore, by~\eqref{ii},~\eqref{geomell} and~\eqref{vest} we are able
to conclude that
$$
H_{ij}(\xi) \eta_i \eta_j = H_{ij}(\xi) (\alpha \xi_i + v_i) (\alpha \xi_j + v_j)
= H_{ij}(\xi) v_i v_j \ge \tilde{c} |v|^2 \ge \tilde{c},
$$
which is~\eqref{END2}.

Now we point out that
\bequ \label{X90}
\begin{split}
\left[ \Hess \,(B \circ H)(\xi) \right]_{i j} \zeta_i \zeta_j & =
\left[ B''(H(\xi)) H_i(\xi) H_j(\xi) + B'(H(\xi)) H_{i j}(\xi) \right]\zeta_i\zeta_j\\
& = (p-1)(H(\xi))^{p-2} \Big( H_i(\xi)\zeta_i\Big)^2
+(H(\xi))^{p-1} H_{i j}(\xi)\zeta_i\zeta_j.
\end{split}
\equ
Moreover, by homogeneity, $H(\xi)\in [c_1 |\xi|, C_1|\xi|]$,
$|\nabla H(\xi)|\le C_1$, and $|H_{ij}(\xi)|\le C_1/|\xi|$,
for suitable $C_1\ge c_1>0$. 
Therefore
\begin{equation}\label{RT1}\begin{split}
\left[ \Hess \,(B \circ H)(\xi) \right]_{i j} \zeta_i \zeta_j
\le 
(p-1)C_1^{p}|\xi|^{p-2}|\zeta|^2
+C_1^p |\xi|^{p-2}|\zeta|^2
.\end{split}\end{equation}
Now we claim that
\begin{equation}\label{RT2}\begin{split}
\left[ \Hess \,(B \circ H)(\xi) \right]_{i j} \zeta_i \zeta_j
\ge c_\star |\xi|^{p-2}|\zeta|^2,
\end{split}\end{equation}
for some~$c_\star>0$.
To prove it, we observe that~$B\circ H$ is homogeneous of degree~$p$,
hence~$\Hess \,(B \circ H)$
is homogeneous of degree~$p-2$, so without loss of generality we may assume~$|\xi|=1$.
Also, we write~$\zeta=\alpha\xi+w$, with~$\alpha\in\R$ and~$w\in\xi^\perp$, so that~$|\zeta|^2=\alpha^2+|w|^2$.
We observe that
\begin{equation}\label{L77}
(H(\xi))^{p-1} H_{i j}(\xi)\zeta_i\zeta_j=(H(\xi))^{p-1} H_{i j}(\xi)w_i w_j
\ge c_1^{p-1} c\,|w|^2,\end{equation}
due to point~$(ii)$ in Lemma~\ref{homain} and~\eqref{END}.
Now, we distinguish two cases. If
$$
|w| \ge c_1 |\zeta|/(2 C_1 + 2 c_1),
$$
then we use~\eqref{X90} and~\eqref{L77},
to obtain
\begin{eqnarray*}
&& \left[ \Hess \,(B \circ H)(\xi) \right]_{i j} \zeta_i \zeta_j\ge 
(H(\xi))^{p-1} H_{i j}(\xi)\zeta_i\zeta_j\ge
\frac{ c_1^{p + 1} c}{4 (C_1 + c_1)^2} \,|\zeta|^2.
\end{eqnarray*}
This proves~\eqref{RT2} in this case. On the other hand, if
$$
|w|< c_1 |\zeta|/(2 C_1 + 2 c_1)\le \min\{ c_1 |\zeta| / (2C_1), \, |\zeta| / 2 \},
$$
we recall point~$(i)$ in Lemma~\ref{homain} and we see that
\begin{align*}         
|H_i(\xi)\zeta_i| & \ge |\alpha H_i(\xi)\xi_i|-|H_i(\xi)w_i|\ge |\alpha H(\xi)|-C_1|w| \ge c_1|\alpha|-C_1|w| \\
& = c_1\sqrt{|\zeta|^2-|w|^2}-C_1|w|\ge c_1 \sqrt{\frac34}\,|\zeta|- c_1 \frac{|\zeta|}{2} = \frac{c_1 (\sqrt{3}-1)}{2} \, |\zeta|.
\end{align*}     
Thus, by~\eqref{X90} and~\eqref{L77},
\begin{align*}
\left[ \Hess \,(B \circ H)(\xi) \right]_{i j} \zeta_i \zeta_j & \ge
(p-1)(H(\xi))^{p-2} \Big( H_i(\xi)\zeta_i\Big)^2\\
& \ge \frac{(p-1) c_1^p (\sqrt{3}-1)^2}{4} \, |\zeta|^2.
\end{align*}
This completes the proof of~\eqref{RT2}.
Then \eqref{RT1} and \eqref{RT2}
establish~\eqref{scope} in this case.

\section{A characterization of~$H$ under assumption~(B)} \label{Hchar}

We present here an exhaustive characterization of the functions~$H$
that satisfy~(B). Indeed, we prove that in this case~$H$ is as in~\eqref{HA}.

Let~$H$ and~$B$ be as in hypothesis~(B).
Notice that, as showed in Section~\ref{S:AUX2}, we have that~$B \in C^2([0, +\infty))$, with
\bequ \label{B''pos}
B''(0) > 0.
\equ
By Lemma~\ref{derBH0}, we know that
$$
\partial_i (B \circ H)(\xi) = 
\begin{cases}
B'(H(\xi)) H_i(\xi) & \mbox{ if } \xi \ne 0 \\
0 & \mbox{ if } \xi = 0.
\end{cases}
$$
Thus, we can proceed to compute the second partial
derivatives of~$B \circ H$ at the origin. Writing as~$e_j$ the~$j$-th
component of the standard basis of~$\R^n$, we get
\bequ \label{d2BH0}
\begin{split}
\partial_{i j}^2 (B \circ H)(0) & =
\lim_{t \rightarrow 0} \frac{B'(H(t e_j)) H_i(t e_j) - 0}{t} =
\lim_{t \rightarrow 0^{\pm}} \frac{B'(|t| H(\pm e_j)) H_i(\pm e_j)}{t} \\
& = \pm \lim_{t \rightarrow 0^{\pm}} \frac{B'(|t| H(\pm e_j))}{|t| H(\pm e_j)} H_i(\pm e_j) H(\pm e_j)
= \pm B''(0) H_i(\pm e_j) H(\pm e_j). 
\end{split}
\equ
Therefore, recalling~\eqref{B''pos}, we may conclude that
the limit exists if and only if
$$
H_i(e_j) H(e_j) = - H_i(-e_j) H(-e_j),
\qquad \mbox{for any } i, j \in \{ 1, \ldots, n \},
$$
or, equivalently,
\bequ \label{oddH2i}
\partial_i \left( H^2 \right) (e_j) =
- \partial_i \left( H^2 \right) (-e_j), \qquad \mbox{for any } i, j \in \{ 1, \ldots, n \}.
\equ
Knowing this, we can check the continuity of the derivatives at the origin.
Since
$$
\partial_{i j}^2 (B \circ H)(\xi) = B''(H(\xi)) H_i(\xi) H_j(\xi) + B'(H(\xi)) H_{i j}(\xi),
$$
for any~$\xi \ne 0$, we compute
\begin{equation*}
\begin{split}
\lim_{\xi \rightarrow 0} \partial_{i j}^2 (B \circ H)(\xi) & =
\lim_{\xi \rightarrow 0} \left[ B''(H(\xi)) - \frac{B'(H(\xi))}{H(\xi)} \right] H_i(\xi) H_j(\xi) \\
& \quad + \lim_{\xi \rightarrow 0} \frac{B'(H(\xi))}{H(\xi)} \big[ H_i(\xi) H_j(\xi) + H(\xi) H_{i j}(\xi) \big] \\
& =: L_1 + L_2.
\end{split}
\end{equation*}
We observe that~$L_1 = 0$, since~$H_i H_j$ is homogeneous of degree~$0$,
and thus bounded, and~$B$ is of class~$C^2$ at the origin. Therefore we get
$$
\lim_{\xi \rightarrow 0} \partial_{i j}^2 (B \circ H)(\xi) = L_2 =
B''(0) \lim_{\xi \rightarrow 0} \big[ H_i(\xi) H_j(\xi) + H(\xi) H_{i j}(\xi) \big],
$$
so that, recalling~\eqref{d2BH0}, the continuity of the second derivatives is equivalent to
$$
\lim_{\xi \rightarrow 0} \big[ H_i(\xi) H_j(\xi) + H(\xi) H_{i j}(\xi) \big] = H_i(e_j) H(e_j).
$$
Rewriting last identity as
\bequ \label{d2=d1}
\lim_{\xi \rightarrow 0} \partial_{i j}^2 \left( \frac{H^2}{2} \right) (\xi)
= \partial_i \left( \frac{H^2}{2} \right) (e_j),
\equ
we notice that, since~$\partial_{i j}^2 \left( H^2 / 2 \right)$
is a homogeneous function of degree~$0$, by~\eqref{d2=d1} it has limit
at the origin and so it is necessarily constant. This means that~$H^2$ is
a polynomial of degree~$2$ and thus
$$
H(\xi) = H_M(\xi) := \sqrt{\langle M \xi, \xi \rangle}, \qquad \mbox{for any } \xi \in \R^n,
$$
with~$M \in \mbox{\normalfont Mat}_n(\R)$ symmetric and positive definite. The function~$H_M$ thus defined is clearly positive homogeneous of degree~$1$ and it satisfies~\eqref{oddH2i}, since it is even.

\section{The Wulff shape: a physical interpretation} \label{wulshaapp}

The convex anisotropy~$H$ we dealt with all along the present work is widely considered in the literature. A particular set is typically related to it: the Wulff shape. Considering the dual function~$H^*$ of~$H$, defined by setting
$$
H^*(x) := \sup_{|\xi| = 1} \frac{\langle \xi, x \rangle}{H(\xi)}, \qquad \mbox{for all } x \in \R^n,
$$
the Wulff shape~$W_H$ of~$H$ is the $1$-sublevel set of~$H^*$, that is
\begin{equation} \label{wulshadef}
W_H := \left\{ x \in \R^n : H^*(x) \le 1 \right\}.
\end{equation}
As shown by the classical Wulff theorem~(see e.g. Theorem~$1.1$ in~\cite{T78}),~$W_H$ is the set which minimizes the anisotropic interfacial energy
$$
\Omega \longmapsto \int_{\partial \Omega} H(\nu(x)) \, d\mathcal{H}^{n - 1}(x),
$$
between all sets~$\Omega$ having the same prescribed volume.

This property is frequently used, for instance, to deduce the equilibrium shape of a crystal, due to the anisotropic nature of the forces there involved. A less common application is described in~\cite[Section~$2$]{T78}, where the author addresses the problem of determining the closed path a trawler should follow in order to enclose a fixed amount of fish in the shortest time. Assuming the fish to be uniformly distributed in the sea and denoting by~$H(\xi)$ the time the sailboat employs to travel, say, one mile in direction~$\xi \in \partial B_1$, it is proved that the optimal path is given by following the frontier of a suitable dilation of the Wulff shape of~$H$.

Next we present another physical interpretation of the Wulff shape, which arises quite naturally in a dynamical model related to our framework. Consider equations~\eqref{eleq} and~\eqref{ODE} in the case~$B(t) = t^2 / 2$, with no forcing terms (i.e. when~$F:=0$) and take the corresponding hyperbolic evolutionary equation
\begin{equation} \label{W.1}
u_{tt} = {\rm div}\,\big( H(\nabla u) \nabla H(\nabla u)\big).
\end{equation}
Notice that~\eqref{W.1} is the classical wave equation when~$H(\xi):=|\xi|$. Then, define~$u^\omega$ to be a one-dimensional travelling wave of velocity~$c_\omega > 0$, that is
$$
u^\omega(x, t) := u_0(\omega \cdot x - c_\omega t),
$$
with~$u_0$ smooth and increasing for simplicity. Using the homogeneity properties of~$H$ (e.g.~\eqref{i} and~\eqref{ii}), we see that, if~$u_0$ is not affine, then it is a solution of~\eqref{W.1} if and only if
$$
c_\omega = H(\omega).
$$

In this setting, the points reached by the plane wave~$u^\omega$ in a unit of time form exactly the set~$\{ x \in \R^n : \omega \cdot x \le H(\omega) \}$. By taking all the possible directions~$\omega \in \partial B_1$ we obtain
$$
\bigcap_{\omega \in \partial B_1} \{ x\in\R^n : \omega \cdot x \le H(\omega) \},
$$
which is the Wulff shape of the velocity function~$\omega \mapsto c_\omega = H(\omega)$, as one can easily check recalling definition~\eqref{wulshadef}.

\section*{Acknowledgements}

We thank Guglielmo Albanese and Simona Scoleri for several interesting conversations.
This work
was supported by the
ERC Grant~$\epsilon$ ({\it Elliptic Pde's and Symmetry of Interfaces and Layers
for Odd Nonlinearities}).

\end{document}